\documentclass[]{scrartcl}

\usepackage{lipsum}
\usepackage{amsthm}
\usepackage{amsmath}
\usepackage{amsfonts}
\usepackage{amssymb}
\usepackage{graphicx}
\usepackage{epstopdf}
\usepackage{algorithm}
\usepackage{algorithmic}
\ifpdf
  \DeclareGraphicsExtensions{.eps,.pdf,.png,.jpg}
\else
  \DeclareGraphicsExtensions{.eps}
\fi

\newcommand{\TheTitle}{MPC, Cost Controllability, and Homogeneity}
\newcommand{\TheAuthors}{J.-M. Coron, L. Gr\"{u}ne, and K. Worthmann}

\title{{Model Predictive Control, Cost Controllability, and Homogeneity}%
\thanks{Jean-Michel Coron acknowledges funding from the Miller Institute and from the Agence Nationale de La Recherche (ANR), grant ANR Finite4SoS. He also thanks the Miller Institute and UC Berkeley for their hospitality. Karl Worthmann gratefully acknowledges funding from the German Research Foundation (Deutsche Forschungsgemeinschaft, grant WO 2056/6-1).}}

\author{
  Jean-Michel Coron\thanks{Sorbonne Universit\'{e}, CNRS, INRIA, Laboratoire Jacques-Louis Lions, \'{e}quipe Cage, Paris, France
    (\textit{coron@ann.jussieu.fr}).}
  \and
  Lars Gr\"{u}ne\thanks{Mathematical Institute, Faculty of Mathematics, Physics and Computer Science, University of Bayreuth, Germany (\textit{lars.gruene@uni-bayreuth.de}).}
  \and
  Karl Worthmann\thanks{Institute for Mathematics, Faculty of Mathematics and Natural Sciences, Technische Universit\"{a}t Ilmenau, Ilmenau, Germany (\textit{karl.worthmann@tu-ilmenau.de}).}
}

\usepackage{amsopn}
\usepackage{newfloat}

\usepackage{hyperref}
\ifpdf
\hypersetup{
  pdftitle={\TheTitle},
  pdfauthor={\TheAuthors}
}
\fi

\usepackage[capitalise]{cleveref}
\crefname{equation}{}{}

\usepackage{color}

\def\R{\mathbb{R}}
\def\eps{\varepsilon}

\newtheorem{theorem}{Theorem}
\newtheorem{proposition}{Proposition}
\newtheorem{lemma}{Lemma}
\newtheorem{definition}{Definition}
\newtheorem{assumption}{Assumption}
\newtheorem{example}{Example}
\newtheorem{remark}{Remark}

\begin{document}

\maketitle

\begin{abstract}
	We are concerned with the design of Model Predictive Control (MPC) schemes such that asymptotic stability of the resulting closed loop is guaranteed %
	--~even if the linearization at the desired set point fails to be stabilizable. %
	Therefore, we propose to construct the stage cost based on the homogeneous approximation and %
	rigorously show that applying MPC yields an asymptotically stable closed-loop behavior if the homogeneous approximation is asymptotically null controllable. %
	To this end, we verify cost controllability --~a condition relating the current state, the stage cost, and the growth of the value function w.r.t.\ time~-- %
	for this class of systems %
	in order to provide stability and performance guarantees for the proposed MPC scheme without stabilizing terminal costs or constraints.
\end{abstract}




\section{Introduction}

\noindent Model Predictive Control (MPC) is nowadays a well-established and widely applied control technique for linear and nonlinear systems, see, e.g.\ the textbooks~\cite{GrunPann17,KouvCann16,RawlMayn17} and the survey papers~\cite{ForbPatw15,Lee11}. One of the main reasons for its success is the simplicity of the underlying idea: measure the current state of the system in consideration, optimize its behaviour on a finite prediction window, and implement the first piece of the computed control function before repeating the procedure ad infinitum. However, the stability analysis (or the construction of stabilizing terminal costs and constraints) is often based on the (explicit or implicit) assumption that the linearization at the desired set point is stabilizable~\cite{ChenAllg98}. Otherwise, the origin may not be asymptotically stable w.r.t.\ the MPC closed loop if, e.g.\ a quadratic stage cost is used --~independently of the length of the optimization horizon as rigorously shown in~\cite{MullWort17} for the mobile robot example, i.e.\ a finite-time controllable system.

We propose to design the stage cost based on the homogeneous approximation~\cite{Kaws88,AndrPral08} and \cite[Section 12.3]{Coro07}. %
Homogeneity --~which is a generalization of linearity to nonlinear systems~-- is a property, %
which was extensively used to construct stabilizing (homogeneous) control laws~\cite{Kawsk90,CoroPral91,Herm95,MoriSams96,Mori99}. %
We show that, if the homogeneous approximation is stabilizable, the desired set point is asymptotically stable w.r.t.\ the MPC closed-loop without terminal costs or constraints. Here, asymptotic stability is understood in a global (homogeneous system with degree zero), semi-global (negative degree of homogeneity), or local (homogeneous approximation) sense. All results are motivated and illustrated by examples.

To this end, we verify \textit{cost controllability} --~a condition relating the growth of the value function w.r.t.\ time with the stage cost. Then, stability guarantees or a desired performance bound for the proposed MPC scheme can be deduced if the prediction horizon is sufficiently large using techniques originally developed in a discrete-time setting, see \cite{Grun09,GrunPann10}, adopted to the cost-controllability condition~\cite{TunaMess06,Wort11}, and transferred to the continuous-time setting~\cite{ReblAllg12,WortRebl14}. While we first show our results for homogeneous systems, we make use of homogeneous Lyapunov functions (see, in particular, \cite{Rosi92,CoroPral91,Grun00,BaccRosi06} and \cite[Section 12.3]{Coro07}) in order to show that our results are locally preserved for (general) nonlinear systems, i.e.\ that cost controllability and local asymptotic stability of the origin for the MPC closed loop hold if the homogeneous approximation at the origin is stabilizable.

The paper is organized as follows: In \cref{sec:ProblemFormulation}, cost controllability and its relationship with MPC are recalled and adapted to our setting. Then, we introduce homogeneous control systems in accordance to~\cite{Grun00} and~\cite{BhatBern05} and give an example that MPC with quadratic costs may fail for finite-time stabilizable systems with zero degree of homogeneity. In \cref{sec:CostControllability}, we establish cost controllability and asymptotic stability of the origin for the MPC closed loop for control systems with zero degree of homogeneity and the local counterparts of these results if the degree of homogeneity is negative. Then, in \cref{sec:HomogeneousApproximation}, we extend our results to systems with globally asymptotically null controllable homogeneous approximation before conclusions are drawn in \cref{sec:conclusions}.

\textbf{Notation}: We call a continuous function $\rho: [0,\infty) \rightarrow [0,\infty)$ of class $\mathcal{K}_\infty$ if it maps zero to zero, is strictly monotonically increasing, and unbounded. For further details on comparison functions, we refer to the compendium~\cite{Kell14} and the references therein. The set $\mathrm{co}\,K$ denotes the convex hull of the set~$K$.

\section{Problem Formulation}\label{sec:ProblemFormulation}

We consider the control system governed by the autonomous ordinary differential equation
\begin{equation}\label{eq:SystemDynamics}
	\dot{x}(t) = f(x(t),u(t))
\end{equation}
with state $x(t) \in \mathbb{R}^n$ and control $u(t) \in \mathbb{R}^m$ at time $t \geq 0$. %
We assume continuity of the map $f: \mathbb{R}^n \times \mathbb{R}^m \rightarrow \mathbb{R}^n$ and local Lipschitz continuity w.r.t.\ its first argument on $\mathbb{R}^n \setminus \{0\}$ %
such that, for each initial value~$x^0$ and each control function $u \in \mathcal{L}^{\infty}_{\operatorname{loc}}([0,\infty),\mathbb{R}^m)$, %
the solution trajectory $x(\cdot;x^0,u)$ uniquely exists (on its maximal interval~$I_{x^0}$ of existence).

\subsection{Model Predictive Control}

Let $f(0,0) = 0$ hold, i.e.~the origin~$0 \in \mathbb{R}^n$ is a (controlled) equilibrium. %
Our goal is the stabilization of system~\cref{eq:SystemDynamics} at the origin. %
As a preliminary condition, we require a stage cost $\ell: \mathbb{R}^n \times \mathbb{R}^m \rightarrow \mathbb{R}_{\geq 0}$ satisfying
\begin{equation}\label{eq:StageCostsCondition}
	\rho(\|x\|) \leq \ell^\star(x) := \inf_{u \in \mathbb{R}^m} \ell(x,u) \leq \varrho(\|x\|) \qquad\forall\,x \in \mathbb{R}^n,
\end{equation}
for $\mathcal{K}_\infty$-functions~$\rho$ and~$\varrho$. %

We briefly introduce model predictive control. The basic idea is simple: %
For the current state $\hat{x} \in \mathbb{R}^n$, optimize a cost functional on a finite prediction horizon~$[0,T]$ %
w.r.t.\ the system dynamics~\cref{eq:SystemDynamics}, the initial condition $x(0) = \hat{x}$, and, if present, state and control constraints. %
Then, the first portion $u^\star(t)$, $t \in [0,\delta)$, of the computed optimal control function $u^\star = u^\star(\cdot;\hat{x}) \in \mathcal{L}^\infty([0,T],\mathbb{R}^m)$ %
is implemented before the prediction window is shifted forward in time. %
The procedure is repeated ad infinitum. 
Assuming existence of a minimizer for simplicity of exposition, the MPC scheme reads as displayed in \cref{alg:MPC}.
\begin{algorithm}
	\caption{Model Predictive Control}
 	\label{alg:MPC}
 	\begin{algorithmic}
 		\STATE{Set $t = 0$ and let the prediction horizon~$T$ and the time shift~$\delta \in (0,T)$ be given.}
 		\begin{enumerate}
 			\STATE{Measure the current state~$\hat{x}:=x(t)$}		
			\STATE{Solve the optimal control problem
			\begin{displaymath}
				\min_{u \in \mathcal{L}^{\infty}([0,T],\mathbb{R}^m)}\ \int_0^T \ell(x(s;\hat{x},u),u(s))\,\mathrm{d}s %
				\quad\text{subject to}\quad x(0;\hat{x},u) = \hat{x} \text{ and~\cref{eq:SystemDynamics}}
			\end{displaymath}
			and denote its solution by $u^\star = u^\star(\cdot;\hat{x})$.}
		\STATE{Implement $u^\star(s)|_{s \in [0,\delta)}$ at the plant, increment~$t$ by $\delta$, and goto Step~1.}
		\end{enumerate}
 	\end{algorithmic}
\end{algorithm}

\noindent We define the value function $V_t:\mathbb{R}^n \rightarrow \mathbb{R}_{\geq 0} \cup \{\infty\}$ by
\begin{equation}\nonumber
	V_t(\hat{x}) :=\inf_{u \in \mathcal{L}^{\infty}([0,t],\mathbb{R}^m)} \int_0^t \ell(x(s;\hat{x},u),u(s))\,\mathrm{d}s \quad\text{subject to}\quad x(0;\hat{x},u) = \hat{x} \text{ and~\cref{eq:SystemDynamics}},
\end{equation}
with the convention that $V_t(\hat{x})=\infty$ if there exists no $u \in \mathcal{L}^{\infty}([0,t],\mathbb{R}^m)$ such that $x(t;\hat{x},u)$, $t \in [0,T]$, is defined. %
This definition allows us to state the following condition relating the stage cost~$\ell$ at the current state~$\hat{x}$ and the value function.
\begin{assumption}[Cost controllability on a set $K$]\label{ass:GrowthCondition_nonglobal}
	Consider a set $K\subseteq \mathbb{R}^n$ containing $0$ in its interior. %
	We assume that there exists a monotonically increasing, bounded function $B_K: [0,\infty) \rightarrow \mathbb{R}_{\geq 0}$, which satisfies the inequality
	\begin{equation}\label{eq:GrowthCondition}
		V_t(x) \leq B_K(t)\ \ell^\star(x) \qquad\forall\,(t,x) \in \mathbb{R}_{\geq 0} \times K.
	\end{equation}
\end{assumption}

For the formulation of the following theorem, which is an extension of a result from \cite{ReblAllg12}, we moreover need the sublevel set
\[
	V_T^{-1}([0,L]) := \{x\in \R^n\,|\, V_T(x) \leq L\}.
\]
\begin{theorem}\label{thm:AsymptoticStabilityMPC_nonglobal}
	Consider the MPC scheme given in \cref{alg:MPC} %
	and let condition~\cref{eq:StageCostsCondition} and \cref{ass:GrowthCondition_nonglobal} for a set $K \subseteq \R^n$ be satisfied %
	for the stage cost~$\ell$ and the system dynamics~\cref{eq:SystemDynamics}. %
	Then, for given $\delta > 0$, there exists a sufficiently large prediction horizon $T > \delta$ %
	such that the origin is locally asymptotically stable for the MPC closed loop. Moreover, the domain of attraction contains all level sets $V_T^{-1}([0,L])\subseteq K$.
	
	If $K = \mathbb{R}^n$, the origin is globally asymptotically stable for the MPC closed loop. %
	Moreover, if the assumptions hold for each compact set $K\subset\R^n$ containing $0$ in its interior %
	(with $B_K$ possibly depending on $K$), %
	then the origin is semiglobally asymptotically stable for the MPC closed loop, i.e., %
	for each compact set $\widehat K\subset \R^n$ and each $\delta > 0$ there is $T>0$, %
	such that the origin is locally asymptotically stable for the MPC closed loop and $\widehat K$ is contained in the domain of attraction.
\end{theorem}
\begin{proof} 
	The assertion for $K = \mathbb{R}^n$ was already shown in \cite{ReblAllg12}. %
	To prove the non-global version, i.e.\ $K \subsetneq \mathbb{R}^n$, %
	we proceed similarly to the discrete time counterpart given in \cite[Theorem 6.35]{GrunPann17} (see also \cite{BoccGrun14} for related results). %
	Hence we only give a sketch.

	\noindent First one shows using the same techniques as in \cite{ReblAllg12} %
	that for sufficiently large $T$ the optimal value function $V_T$ decreases along solutions in $K$. %
	This implies that each level set $V^{-1}([0,L])\subset K$ is forward invariant. %
	On this level set asymptotic stability with domain of attraction $V^{-1}([0,L])$ follows with usual Lyapunov function arguments. %
	The second statement follows by estimating the size of the level sets $V_T^{-1}([0,L])$ exactly as in the proof of \cite[Theorem 6.35]{GrunPann17}.
\end{proof}

\noindent \cref{ass:GrowthCondition_nonglobal} is a sufficient condition for stability that is particularly suitable for the class of systems and cost functions that we will investigate in this paper. Of course, for other systems and cost functions other sufficient conditions may be better suited. Two examples will be discussed in \cref{rem:conditions} at the end of \cref{sec:CostControllability}.

Monotonicity of the (growth) function $B_K$ in \cref{ass:GrowthCondition_nonglobal} can be assumed w.l.o.g.\ %
since the left hand side of inequality~\cref{eq:GrowthCondition} is monotonically increasing due to the absence of a terminal cost. %
Furthermore, the condition $\limsup_{t \rightarrow \infty} B_K(t)/t < 1$ suffices %
to ensure the existence of a bounded (growth) function~$B_K$, see~\cite[Proposition\,6]{MullWort17}. %
Using the abbreviation $\gamma := \limsup_{t \rightarrow \infty} B_K(t)$, \cref{ass:GrowthCondition_nonglobal} ensures the inequality
\begin{equation}\nonumber
	\sup_{\hat{x} \in K \setminus \{0\}} \left( \frac {V_t(\hat{x})}{\ell^\star(\hat{x})} \right) \leq \gamma \qquad\forall\,t \geq 0.
\end{equation}
In particular, this implies that the value function~$V_t$ cannot contain terms of lower order than the \textit{minimized} stage cost $\ell^\star(\hat{x}) = \inf_{u \in \mathbb{R}^m} \ell(\hat{x},u)$. Otherwise \cref{ass:GrowthCondition_nonglobal} cannot hold in a neighbourhood of the origin and, thus, on the set $K$.

\subsection{Homogeneous Systems and quadratic stage costs}

The following definition of homogeneity is taken from~\cite{Grun00} with a slightly adapted notation. 
\begin{definition}[Homogeneous System]\label{def:Homogeneity}\, %
	For given $r=(r_1,\ldots,r_n)\in (0,+\infty)^n$, $s = (s_1, \ldots, s_m) \in (0,+\infty)^m$, and $\tau \in (-\min_{i} r_i, \infty)$, System~\cref{eq:SystemDynamics} is said to be $(r,s,\tau)$-homogeneous if
	\begin{equation}\label{NotationHomogeneityCondition}
		f(\Lambda_\alpha x,\Delta_\alpha u) = \alpha^\tau \Lambda_\alpha f(x,u) \qquad\forall\,(x,u) \in \mathbb{R}^n \times \mathbb{R}^m \text{ and } \alpha \geq 0
	\end{equation}
	is satisfied, where the matrices $\Lambda_\alpha \in \mathbb{R}^{n \times n}$ and $\Delta_\alpha \in \mathbb{R}^{m \times m}$ are defined as
	\begin{equation}\nonumber
		\Lambda_\alpha = \left( \begin{array}{cccc}
			\alpha^{r_1} & 0 & \cdots & 0 \\
			0 & \ddots & \ddots & \vdots \\
			\vdots & \ddots & \ddots & 0 \\
			0 & \cdots & 0 & \alpha^{r_n}
		\end{array} \right) \qquad\text{and}\qquad \Delta_\alpha = \left( \begin{array}{cccc}
			\alpha^{s_1} & 0 & \cdots & 0 \\
			0 & \ddots & \ddots & \vdots \\
			\vdots & \ddots & \ddots & 0 \\
			0 & \cdots & 0 & \alpha^{s_m}
		\end{array} \right).
	\end{equation}
\end{definition}

\noindent Homogeneity can be considered as a generalization of linearity. For linear systems, i.e.\ $\dot{x}(t) = Ax(t) + Bu(t)$, homogeneity is trivially satisfied with $\tau = 0$ and $r_i = 1 = s_j$ for $(i,j) \in \{1,\ldots,n\} \times \{1,\ldots,m\}$. Using \cref{def:Homogeneity}, one easily obtains the identity
\begin{equation}\label{eq:HomogeneityTrajectory}
	x(t;\Lambda_\alpha x^0, \Delta_\alpha u(\alpha^\tau \cdot)) = \Lambda_\alpha x(\alpha^\tau t;x^0,u)
\end{equation}
for all $t \in I_{x^0} \cap [0,\infty)$. For $\tau = 0$, the identity~\cref{eq:HomogeneityTrajectory} simplifies to $x(t;\Lambda_\alpha x^0, \Delta_\alpha u) = \Lambda_\alpha x(t;x^0,u)$.

We consider the following example in order to show that the design of stage cost in MPC matters.
\begin{example}\label{ExampleNonholonomicRobotHomApprox}
	Let the system dynamics~\cref{eq:SystemDynamics} be given be
	\begin{equation}\label{eq:ExampleNonholonomicRobotHomApprox}
		\dot{x}(t) = f(x(t),u(t)) := \left( \begin{array}{c} u_1(t) \\ x_3(t) u_1(t) \\ u_2(t) \end{array} \right) %
		= \left( \begin{array}{c} 1 \\ x_3(t) \\ 0 \end{array} \right) u_1(t) + \left( \begin{array}{c} 0 \\ 0 \\ 1 \end{array} \right) u_2(t).
	\end{equation}
\end{example}

\noindent \cref{ExampleNonholonomicRobotHomApprox} is governed by homogeneous system dynamics since Condition~\cref{NotationHomogeneityCondition} holds with $\tau = 0$, $r_2 = 2$, and $r_1=r_3=s_1=s_2 = 1$. Moreover, the system is null controllable in finite time, which means that starting from any given state one can steer the control system to $0 \in \mathbb{R}^3$ in finite time (even in arbitrary small positive time). This property follows from the Rashevski-Chow Theorem~\cite[Theorem 3.18]{Coro07}. It can also be checked directly as follows.
Ensure that the $x_3$-component is not equal to $0$ by using $u_2$. Then take $u_2=0$ (so $x_3$ does not change) and use $u_1$ to steer the $x_2$-component to $0$. Then take $u_1=0$ (so $x_2$ remains equal to $0$) and set the $x_3$-component equal to $0$ using $u_2$. Finally, take $u_2=0$ (so both $x_2$ and $x_3$ remain equal to $0$) and use $u_1$ to steer the $x_1$-component to $0$. %
It follows from this method to steer the control system \cref{eq:ExampleNonholonomicRobotHomApprox} to $0\in \mathbb{R}^3$ that the value function~$V_t$ is uniformly bounded w.r.t.\ $t \in \mathbb{R}_{\geq 0}$ on compact sets of $\mathbb{R}^3$. %
Nevertheless, one can show that there does not exist a (finite) prediction horizon %
such that the origin is asymptotically stable w.r.t.\ the MPC closed loop. %
\begin{proposition}
	We consider the MPC scheme presented in \cref{alg:MPC} %
	for the combination of system~\cref{eq:ExampleNonholonomicRobotHomApprox} and the quadratic stage cost $\ell(x,u) = x^\top Q x + u^\top R u$ %
	with positive definite matrices $Q \in \mathbb{R}^{3 \times 3}$ and $R \in \mathbb{R}^{2 \times 2}$. %
	Then, for arbitrary but fixed prediction horizon~$T$, the origin is not (locally) asymptotically for the MPC closed loop.
\end{proposition}
\begin{proof}
	Firstly, it can be shown that, for every $\varepsilon > 0$, %
	there exists an initial value $x^0 \in \{x \in \mathbb{R}^3 \mid \| x \| \leq \varepsilon\}$ such that $u^\star \equiv 0$ is the unique optimal control. %
	The proof is almost literally the same as its analogon for the nonholonomic robot presented in~\cite[Subsection~4.1]{MullWort17} if %
	Condition~(13) is replaced by $(Qx^0)_3 = 0 = (Qx^0)_1 + (Qx^0)_2 \theta^0$ in view of the slightly different adjoint equation. %
	Hence, the closed-loop trajectory emanating at $x^0$ stays at $x^0$ forever and does, thus, not converge to the origin, which shows the assertion.
\end{proof}
We note that this implies that \cref{ass:GrowthCondition_nonglobal} must be violated. %
In the following, our goal is to construct a stage cost %
such that asymptotic stability of the origin w.r.t.\ the MPC closed loop %
is ensured for a sufficiently long prediction horizon~$T$ for a class of systems %
including both \cref{ExampleNonholonomicRobotHomApprox} and the mentioned nonholonomic robot example %
provided that the homogeneous (approximation of the) system is globally asymptotically null controllable. %
Moreover, we want to keep the construction simple by solely replacing the quadratic performance measure by one based on the \textit{dilated norm} $N: \mathbb{R}^n \rightarrow [0,\infty)$, i.e. 
\begin{equation}\label{NotationDilatedNorm}
	N(x) := \left( \sum_{i=1}^n x_i^{\frac d{r_i}} \right)^{\frac 1d} \qquad\text{ with }\qquad d := 2 \prod_{i=1}^n r_i.
\end{equation}
Note that the \textit{dilated norm} satisfies the equality
\begin{equation}\label{eq:HomogeneityDilatedNormInterplay}
	N(\Lambda_\alpha x) = \sqrt[d]{\sum_{i=1}^n \left( \alpha^{r_i} x_i \right)^{\frac d {r_i}}} = \sqrt[d]{\alpha^d \sum_{i=1}^n x_i^{\frac d {r_i}}} = \alpha N(x) \qquad\forall\, (\alpha,x) \in \mathbb{R}_{\geq 0} \times \mathbb{R}^n
\end{equation}
and that it is positive on $\mathbb{R}^n\setminus\{0\}$. However it is not a norm unless $\Lambda_\alpha = \alpha I$.


%
%

\section{Cost Controllability for Homogeneous Systems}\label{sec:CostControllability}

We consider homogeneous systems and %
propose to define the stage cost~$\ell: \mathbb{R}^n \times \mathbb{R}^m \rightarrow \mathbb{R}_{\geq 0}$ by
\begin{equation}\label{eq:StageCosts}
	\ell(x,u) := N(x)^d + \sum_{j=1}^m u_j^{\frac d{s_j}} = \sum_{i=1}^n x_i^{\frac d{r_i}} + \sum_{j=1}^m u_j^{\frac d{s_j}}.
\end{equation}
The stage cost is based on the \textit{dilated norm} and the coefficients $r_i$, $i \in \{1,\ldots,n\}$, and $s_j$, $j \in \{1,\ldots,m\}$, from \cref{def:Homogeneity}. %
Note that the definition of $\ell$ in~\cref{eq:StageCosts} leads to quadratic stage cost for linear systems $\dot{x}(t) = Ax(t) + Bu(t)$. %
Moreover, the definition is compatible with homogeneity analogously to~\cref{eq:HomogeneityDilatedNormInterplay}.
\begin{proposition}\label{prop:StageCosts}
	The stage costs~$\ell$ are homogeneous, i.e.
	\begin{equation}\label{eq:HomogeneityStageCostsInterplay}
		\ell(x(s;\Lambda_\alpha x^0,\Delta_\alpha u(\alpha^\tau \cdot)),\Delta_\alpha u(\alpha^\tau s)) = \alpha^d \ell(x(\alpha^\tau s;x^0,u),u(\alpha^\tau s))
	\end{equation}
	holds for $\alpha \geq 0$. For $\tau = 0$, equality~\cref{eq:HomogeneityStageCostsInterplay} simplifies to
	\begin{equation}\label{eq:HomogeneityStageCostsInterplayTauEqualZero}
		\ell(x(s;\Lambda_\alpha x^0,\Delta_\alpha u),\Delta_\alpha u(s)) = \alpha^d \ell(x(s;x^0,u),u(s)).
	\end{equation}
\end{proposition}
\begin{proof} A direct calculation shows
	\begin{align}
		\ell(x(s;\Lambda_\alpha x^0,\Delta_\alpha u(\alpha^\tau \cdot)),\Delta_\alpha u(t)) & = N(\underbrace{x(s;\Lambda_\alpha x^0, \Delta_\alpha u(\alpha^\tau \cdot))}_{\stackrel{\cref{eq:HomogeneityTrajectory}}{=} \Lambda_\alpha x(\alpha^\tau s;x^0,u)})^d + \sum_{j=1}^m \left( \alpha^{s_j} u_{j}(t) \right)^{\frac d {s_j}} \nonumber \\
		& \stackrel{\cref{eq:HomogeneityDilatedNormInterplay}}{=} \alpha^d N(x(\alpha^\tau s;x^0,u))^{d} + \alpha^d \sum_{j=1}^m u_{j}(t)^{\frac d {s_j}} \nonumber \\
		& = \alpha^d\, \ell(x(\alpha^\tau s;x^0,u),u(t)), \nonumber
	\end{align}
	i.e.~the assertion if $t$ is replaced by $\alpha^\tau s$. 
\end{proof}
\noindent Our goal in this section is to verify cost controllability for homogeneous, globally asymptotically null controllable systems. %
\begin{definition}\label{def:AsymptoticStabilizability}
	The system~\cref{eq:SystemDynamics} is said to be globally asymptotically null controllable if, for each $x^0 \in \mathbb{R}^n$, there exists a control function $u^0 \in \mathcal{L}^{\infty}_{\operatorname{loc}}([0,\infty),\mathbb{R}^m)$ such that $\lim_{t \rightarrow \infty} \|x(t;x^0,u^0)\| = 0$ holds.
\end{definition}
\noindent Note that a locally asymptotically null controllable homogeneous system is always globally asymptotically null controllable.

Inequality~\cref{eq:GrowthCondition} is equivalent to
\begin{equation}\label{eq:GrowthConditionInequalityHomogeneousSystems}
	V(t) \leq B_K(t) N(x^0)^d \qquad\forall\,(t,x^0) \in \mathbb{R}_{\geq 0} \times K,
\end{equation}
using the stage cost~\cref{eq:StageCosts}. %
Under this condition, asymptotic stability of the origin w.r.t.\ the MPC closed loop is ensured for a sufficiently large prediction horizon~$T$ using \cref{thm:AsymptoticStabilityMPC_nonglobal} since Condition~\cref{eq:StageCostsCondition} is satisfied.

Before we proceed, we consider a simple example in order to illustrate and motivate our results.
\begin{example}\label{Example:OneDimensional}
	Let the system dynamics~\cref{eq:SystemDynamics} be given by
	\begin{equation}\label{sys-1dk}
		\dot{x}(t) = |x(t)|^k \textrm{sign}(x) + u \qquad\text{with $k \in (0, +\infty)$.}
	\end{equation}
	Since we have
	\begin{eqnarray*}
		f(\Lambda_\alpha x, \Delta_\alpha u) & = & %
		|\alpha^r x|^k\text{sign}(x) + \alpha^s u, \\
		\alpha^\tau \Lambda_\alpha f(x,u) & = & \alpha^{\tau+r}(|x|^k\text{sign}(x) + u),
	\end{eqnarray*}
	we get the $(r,s,\tau)$-homogeneity of \cref{sys-1dk} if and only if $\tau + r = rk$ and $\tau + r = s$. This leads to the following three cases:
	\begin{itemize}
		\item $k = 1$: Degree of homogeneity zero, i.e.\ $\tau = 0$ with coefficients $r = s$
		\item $k \in (0,1)$: Negative degree of homogeneity, i.e.\ $\tau < 0$ with coefficients $r = s/k$
		\item $k > 1$: Positive degree of homogeneity, i.e.\ $\tau > 0$ with $r = s/k$
	\end{itemize}
	Next, we compute the infinite horizon optimal value function using the Hamilton-Jacobi-Equation
	\begin{equation}\label{MotivatingExampleWithInputHJB}
		\min_{u \in \mathbb{R}^m} \Big( \langle \nabla V(x),f(x,u) \rangle + \ell(x,u) \Big) = 0
	\end{equation}
	with $m = 1$. To this end, we quadratically penalize the control effort to simplify the following calculations %
	by using the stage cost $\ell: \mathbb{R}^n \times \mathbb{R}^m$, $n =m = 1$, defined by $\ell(x,u) = (x^2)^d + u^2$, %
	which yields $u^\star = - \frac 12 V^\prime(x)$. %
	Then, plugging this expression into~\cref{MotivatingExampleWithInputHJB}, i.e.\ $V^\prime(x) \left( |x|^k\text{sign}(x) + u^\star \right) + \ell(x,u^\star) = 0$, %
	leads to $V^\prime(x) |x|^k\text{sign}(x) - \frac 14 V^\prime(x)^2 + (x^2)^d = 0$, or %
	(only the root having the sign of $x$ has to be taken into account since $V$ is increasing on $[0,+\infty)$ and decreasing on $(-\infty,0]$)
	\begin{equation}\nonumber
		V^\prime(x) = 2 |x|^k \text{sign}(x) + \text{sign}(x)\sqrt{4 |x|^{2k} + 4 (x^2)^d } \stackrel{d = k}{=} 2(1+\sqrt{2}) |x|^k \text{sign}(x).
	\end{equation}
	Overall, we get the value function
	\begin{equation}\nonumber
		V(x) = \frac {2(1 + \sqrt{2})}{k+1} \ |x|^{k+1} \qquad\text{ and, thus, }\qquad \frac {V(x)}{\ell^\star(x)} = \frac {2(1 + \sqrt{2})}{k+1} \ |x|^{1-k}.
	\end{equation}
	In conclusion, we get for
	\begin{itemize}
		\item $k=1$: The growth bound is uniformly upper bounded by $1+\sqrt{2}$, %
			i.e.\ cost controllability in the sense of \cref{ass:GrowthCondition_nonglobal} with $K = \mathbb{R}^n=\mathbb{R}$.
		\item $k \in (0,1)$: The growth bound is uniformly upper bounded on each compact set~$K$ by $\frac {2(1+\sqrt{2})}{k+1} \max_{x \in K} \left\{ |x|^{1-k} \right\}$, %
			i.e.\ cost controllability in the sense of \cref{ass:GrowthCondition_nonglobal}.
		\item $k > 1$: the quotient $V(x)/\ell^\star(x)$ is uniformly upper bounded on $\mathbb{R} \setminus (-\varepsilon, \varepsilon)$ %
			for arbitrary but fixed $\varepsilon > 0$ by $\frac {2(1+\sqrt{2})}{k+1} \max_{x \in \mathbb{R} \setminus (-\varepsilon,\varepsilon)} \left\{ |x|^{1-k} \right\}$.
	\end{itemize}
	Moreover, note that the choice $d=1$, i.e.\ purely quadratic costs, %
	implies for $k \in (0,1)$ that the value function still exhibits a term of order $1+k$ but $\ell^\star(x) = x^2$ holds. %
	Hence, there does not exist a uniform bound for the quotient $V(x) / \ell^\star(x)$ for $x \rightarrow 0$, $x \neq 0$.	
\end{example}
\begin{remark}\label{RemarkDegreeHomogeneity}
	In general, the degree of homogeneity ($\tau < 0$, $\tau = 0$, and $\tau > 0$) is not unique, %
	which can be seen from the (driftless) system~\cref{ExampleNonholonomicRobotHomApprox}. %
	Here, Condition~\cref{NotationHomogeneityCondition} is also satisfied for the combinations
	\begin{equation*}
		\text{$\tau = 0.5$, $r_1=r_3=1$, $r_2=2$, and $s_1 = s_2 = 1.5$}
	\end{equation*}
	and
	\begin{equation*}
		\text{$\tau = -0.5$, $r_1=r_3=1$, $r_2=2$, and $s_1 = s_2 = 0.5$}.
	\end{equation*}	
	We emphasize that the relative weighting of the control in comparison to the penalization of deviations w.r.t.\ the state depends on the choice, %
	which explains the different behavior of the MPC closed loop. For control affine systems, we set $s_j = 2$ whenever possible in view of the Hamilton-Jacobi-Equation.
\end{remark}

\subsection{Degree of homogeneity zero}
In this subsection, we derive cost controllability and, then, conclude asymptotic stability for globally asymptotically null controllable systems with degree of homogeneity $\tau = 0$. %
\begin{theorem}[Cost Controllability]\label{thm:CostControllability}
	Let the system~\cref{eq:SystemDynamics} be globally asymptotically null controllable and $(r,s,\tau=0)$-homogeneous. %
	In addition, let the stage cost be defined by~\cref{eq:StageCosts}. Then, cost controllability, i.e.~inequality~\cref{eq:GrowthConditionInequalityHomogeneousSystems}, holds. %
	In particular, given $\delta>0$, for any sufficiently large $T>0$ the origin is globally asymptotically stable for the MPC closed loop.
\end{theorem}
\begin{proof}
	First, we show inequality~\cref{eq:GrowthConditionInequalityHomogeneousSystems} only for initial values~$x^0$ contained in the compact set
	\begin{equation} \nonumber
		\mathcal{N}_{c_1,c_2} := \{ x \in \mathbb{R}^n \mid c_1 \leq N(x)^d \leq c_2 \}
	\end{equation}
	with $0 < c_1 < c_2 < \infty$. Later on, we show that the same growth function~$B$ can be used for arbitrary initial values $x^0 \in \mathbb{R}^n$.

	Global asymptotical null controllability implies that, for each $x^0 \in \mathcal{N}_{c_1,c_2}$, %
	there exists a control function $
	u^0 \in \mathcal{L}^{\infty}_{\operatorname{loc}}([0,\infty),\mathbb{R}^m)$ %
	such that $\lim_{t \rightarrow \infty} x(t;x^0,
	u^0) = 0$ holds. %
	Hence, continuity of the solution trajectory $x(\cdot;x^0,
	u^0)$ implies the existence of a time instant $
	T^0 \in (0,\infty)$ (depending on $x^0$) such that
	\begin{equation}\nonumber
		N(x(T^0;x^0,u^0))^d = \alpha\ \frac {c_1+c_2}2
	\end{equation}
	where $\alpha \in (0,1)$ is chosen such that $\alpha (c_1+c_2)/2 < c_1$ holds, %
	i.e.~$x(T^0;x^0,u^0)$ is contained in the sublevel set $\{x \in \mathbb{R}^n \mid N(x)^d < c_1\}$. %
	Due to continuity and compactness of the time interval $[0,T^0]$, there exists $\delta^0 > 0$ sufficiently small (depending on $x^0$) such that the chain of inequalities
	\begin{equation}\label{ineq:contraction}
		\alpha c_1 < N(x(T^0;\bar{x},u^0))^d < \alpha c_2
	\end{equation}
	holds for all $\bar{x} \in \mathcal{B}_{\delta^0}(x^0)$. %
	Since $\bigcup_{x^0 \in \mathcal{N}_{c_1,c_2}} \mathcal{B}_{\delta^0}(x^0)$ is an open cover of the compact set $\mathcal{N}_{c_1,c_2}$, %
	there exists a finite subcover and, thus, there exist finitely many \textit{initial values} $x^1,\ldots,x^k$, $k \in \mathbb{N}$, %
	with corresponding $u^i$, $T^i$ and $\delta^i$, satisfying
	\begin{equation}\nonumber
		\mathcal{N}_{c_1,c_2} \subset \bigcup_{i=1}^k \mathcal{B}_{\delta^i}(x^i).
	\end{equation}
	Let us set $T^\star := \min_{i=1,\ldots,k} 
	T^i \in (0,\infty)$. Then, the following function is obviously a growth function~$B$ on the finite interval $[0,T^\star]$ for all $x^0 \in \mathcal{N}_{c_1,c_2}$:
	\begin{align}
		B(t) := & \max_{i=1,\ldots,k} \hspace*{0.15cm} \sup_{\bar{x} \in \mathcal{B}_{\delta^i}(x^i) \cap\, \mathcal{N}_{c_1,c_2}} \hspace*{-0.15cm} N(\bar{x})^{-d} %
		\int_0^t \ell(x(s;\bar{x},u^i),u^i(s))\,\mathrm{d}s \qquad\forall\,t \in [0,T^\star), \nonumber \\
		B(T^\star) := & \max_{i=1,\ldots,k} \hspace*{0.15cm} \sup_{\bar{x} \in \mathcal{B}_{\delta^i}(x^i) \cap\, \mathcal{N}_{c_1,c_2}} \hspace*{-0.15cm} N(\bar{x})^{-d} %
		\int_0^{T^i} \ell(x(s;\bar{x},u^i),u^i(s))\,\mathrm{d}s. \label{defBT*}
	\end{align}
	Note that the preceding line of argumentation implicitly implies well-posedness and boundedness of~$B$ on $[0,T^\star]$ %
	in view of the fact that $N(\bar{x})$ is uniformly bounded from below by $c_1$.\footnote{Every trajectory $x(\cdot;x^i,u^i)$ is bounded on $[0,T^\star]$. %
	In addition, $u^i$ is essentially bounded. Hence, the stage costs are also bounded. %
	Then, continuity of the solution w.r.t.\ the initial value and continuity of the stage cost imply the assertion.} %
	This construction is quite conservative and does not rely on homogeneity. %
	But it is essential to have a bounded~$B$ on a compact time interval of positive length for the following line of argumentation.
	
	Before we proceed, let us present some preliminary considerations. Let $\tilde{x}^0 \notin \mathcal{N}_{c_1,c_2}$ be arbitrary but fixed. %
	Then, if $\tilde{x}^0 \neq 0$, there exists a scaling factor $c \in (0,\infty)$ such that $\Lambda_c \tilde{x}^0 \in \mathcal{N}_{c_1,c_2}$. %
	This implies the existence of an index~$j \in \{1,2,\ldots,k\}$, such that %
	$\Lambda_c \tilde{x}^0 \in \mathcal{B}_{
	\delta^{j}}(x^{j})$. Hence, using $\tilde{x}^0 = \Lambda_{c^{-1}} (\Lambda_c \tilde{x}^0)$, we get
	\begin{align}
		\ell(x(s;\tilde{x}^0,\Delta_{c^{-1}}u^j),\Delta_{c^{-1}}u^j(s)) %
		& \stackrel{\cref{eq:HomogeneityStageCostsInterplayTauEqualZero}}{=} c^{-d} \ell(x(s;\Lambda_c \tilde{x}^0,u^j),u^j(s))\label{eq:HomogeneityIdentity}
	\end{align}
	and, for $\ell^\star(\tilde{x}^0) := \inf_{u \in \mathbb{R}^m} \ell(\tilde{x}^0,u)$,
	\begin{equation}
		\ell^\star(\tilde{x}^0) = N(\tilde{x}^0)^d = N(\Lambda_{c^{-1}} \Lambda_{c} \tilde{x}^0)^d \stackrel{\cref{eq:HomogeneityDilatedNormInterplay}}{=} c^{-d} N(\Lambda_c \tilde{x}^0)^d = c^{-d} \ell^\star(\Lambda_c \tilde{x}^0). \label{eq:EllStar}
	\end{equation}
	Hence, showing the desired inequality (see \cref{ass:GrowthCondition_nonglobal}) for $x^0 \in \mathcal{N}_{c_1,c_2}$ suffices.
	
	The identity~\cref{eq:HomogeneityIdentity} is also the key ingredient to show the assertion for $t > T^\star$ based on the following contraction argument. %
	For each $x^0 \in \mathcal{N}_{c_1,c_2}$, there exists an index $q \in \{1,\ldots,k\}$, and a time instant $t^0 \in [T^\star,\max_{\{i=1,\ldots,k\}} T^i]$ %
	such that $x(t^0;x^0,u^q) \in \alpha \mathcal{N}_{c_1,c_2}$ and
	\begin{equation}\nonumber
		\int_0^{t} \ell(x(s;x^0,u^q),u^q(s))\,\mathrm{d}s \leq B(\min\{t,T^\star\})\, N(x^0)^d \qquad\forall\,t \in [0,t^0]
	\end{equation}
	hold. 
	Next, using the identity~\cref{eq:HomogeneityIdentity} for $\tilde{x}^0 = x(t^0;x^0,u^q)$ with $c = \alpha^{-1} > 1$ %
	allows us to repeat the presented line of reasoning to show that the stage cost are scaled with $\alpha^d$. In addition,
	\begin{equation}\nonumber
		x(T^j;\tilde{x}^0,\Delta_\alpha u^j) = \Lambda_\alpha \underbrace{x(T^j; \Lambda_{\alpha^{-1}} \tilde{x}^0,u^j)}_{\in \alpha \mathcal{N}_{c_1,c_2}} \in \alpha^2 \mathcal{N}_{c_1,c_2}.
	\end{equation}		
	 Hence, we get the desired growth condition if we periodically extend the growth function~$B$ on $(T^\star,\infty)$ by setting
	\begin{equation}\nonumber
		B(t) := \sum_{i=0}^{\lfloor t/T^\star \rfloor - 1} \alpha^{id} B(T^\star) + \alpha^{d \lfloor t/T^\star \rfloor} B(t- T^\star \lfloor t/T^\star \rfloor) \qquad\forall\,t > T^\star,
	\end{equation}
	which is bounded by $\sum_{i=0}^\infty \alpha^{id} B(T^\star) = (1-\alpha^d)^{-1} B(T^\star) \in (0,\infty)$. This completes the proof of the first statement. Asymptotic stability then follows from \cref{thm:AsymptoticStabilityMPC_nonglobal}.
\end{proof}

\subsection{Degree of homogeneity not equal to zero}

We begin with globally asymptotically null controllable systems with negative degree of homogeneity, i.e.\ $\tau < 0$. %
Here, we show semiglobal asymptotic stability of the origin w.r.t.\ the MPC closed loop, i.e.\ local asymptotic stability of the origin with an arbitrary but fixed compact set being the domain of attraction, cf.\ \cref{thm:AsymptoticStabilityMPC_nonglobal}. We refer to \cref{Example:OneDimensional}, which shows the necessity of this restriction for the derivation of cost controllability. 

In \cref{TheoremCostControllabilityHomogeneousNegative} we extend \cref{thm:CostControllability} to homogeneous systems with negative degree.
\begin{proposition}[Cost Controllability for $\tau < 0$]\label{TheoremCostControllabilityHomogeneousNegative}
	Let the system \cref{eq:SystemDynamics} be globally asymptotically null controllable and $(r,s,\tau)$-homogeneous with $\tau < 0$. In addition, let the stage cost be defined by \cref{eq:StageCosts}.
	Then, \cref{eq:GrowthCondition}, i.e.\ cost controllability, holds on any compact set $K \subset \mathbb{R}^n$ with $0 \in {\mathrm int}\, K$. 	
	In particular, given $\delta>0$, for any sufficiently large $T>0$ the origin is semiglobally asymptotically stable for the MPC closed loop.
\end{proposition}
\begin{proof}
	Since the set~$K$ is compact there exists $c_2 \in (0,\infty)$ such that
	\begin{equation}\nonumber
		K \subseteq \{x^0 \in \mathbb{R}^n \mid N(x^0)^d \leq c_2\}
	\end{equation}		
	holds. We show that inequality~\cref{eq:GrowthConditionInequalityHomogeneousSystems} holds for all $x^0$ contained in this sublevel set. To this end, we start analogously to the verification of \cref{thm:CostControllability} with $c_1 \in (0,c_2)$ but define the growth function $B: \mathbb{R}_{\geq 0} \rightarrow \mathbb{R}_{\geq 0}$ differently. Let $\widetilde{B}$ denote the growth function defined analogously as in the proof of \cref{thm:CostControllability} and $\alpha$, $\alpha \in (0,1)$, be an arbitrarily chosen but fixed contraction parameter. Then, the adapted growth function~$B$ is given by
	\begin{equation}
		B(t) := \widetilde{B}(\alpha^{\tau} t) \qquad\text{for all $t \geq 0$.}
	\end{equation}
	Note that $\alpha^{\tau} > 1$ holds since $\alpha \in (0,1)$ and $\tau < 0$. Since $\widetilde{B}$ is monotonically increasing, $B(t) \geq \widetilde{B}(t)$ holds for all $t \geq 0$.\footnote{This preliminary step is important to 
	uniformly (independent of the initial value) reparametrize the time argument in the former growth bound by the factor~$\alpha^{\tau}$} 
	
	Next, we show that inequality~\cref{eq:GrowthConditionInequalityHomogeneousSystems} holds for all $t \geq 0$ for $x^0 \in \mathcal{N}_{c_1,c_2}$. %
	We choose $x^1,\ldots,x^k$ with corresponding $u^i$, $T^i$ and $\delta^i$, $i=1,\ldots,k$, as in the proof of \cref{thm:CostControllability}. %
	For given initial state $x^0 \in \mathcal{N}_{c_1,c_2}$ %
	there exists an index~$q \in \{1,2,\ldots,k\}$ such that $\tilde{x}^0 := x(T^q;x^0,u^q) \in \alpha \mathcal{N}_{c_1,c_2}$ holds. %
	In addition, there exits an index $j \in \{1,2,\ldots,k\}$ %
	such that $\Lambda_{\alpha^{-1}} \tilde{x}^0 \in \mathcal{B}_{\delta^j}(x^j)$ holds and the analogon of identity~\cref{eq:HomogeneityIdentity} for $\tau < 0$ is
	\begin{align}
		\ell(x(s;\tilde{x}^0,\Delta_{c^{-1}}u^j(c^{-\tau}\cdot)),\Delta_{c^{-1}}u^j(c^{-\tau}s)) %
		& \stackrel{\cref{eq:HomogeneityStageCostsInterplay}}{=} c^{-d} \ell(x(c^{-\tau}s;\Lambda_c \tilde{x}^0,u^j),u^j(c^{-\tau}s)). \nonumber
	\end{align}
	Next, we use the scaling factor $c := \alpha^{-1} > 1$ for $\tilde{x}^0$. %
	Note that equation~\cref{eq:EllStar} yields $\ell^\star(\tilde{x}^0) = \alpha^d \ell^\star(\Lambda_{\alpha^{-1}} \tilde{x}^0)$. %
	Overall, for $t \leq \alpha^{-\tau} T^j$, we get 	
	\begin{align}
		\frac {V_t(\tilde{x}^0)}{\ell^\star(\tilde{x}^0)} & = \alpha^{-d}\, \frac {V_{t}(\tilde{x}^0)}{\ell^\star(\Lambda_{\alpha^{-1}} \tilde{x}^0)} \nonumber \\
		& \leq \frac {\alpha^{-d}}{\ell^\star(\Lambda_{\alpha^{-1}} \tilde{x}^0)}\, %
		\int_0^{t} \underbrace{\ell(x(s;\tilde{x}^0,\Delta_{\alpha}u^j(\alpha^{\tau}\cdot)),\Delta_{\alpha}u^j(\alpha^{\tau}s))}_{= \ell(x(\alpha^{\tau}s;\Lambda_{\alpha^{-1}} \tilde{x}^0,u^j),u^j(\alpha^{\tau}s))}\,\mathrm{d}s \nonumber \\
		& = \alpha^{-d} \alpha^{- \tau}\, \frac { \int_0^{\alpha^{\tau}t} \ell(x(\tilde{s};\Lambda_{\alpha^{-1}} \tilde{x}^0,u^j), u^j(\tilde{s}))\,\mathrm{d}\tilde{s}}{\ell^\star(\Lambda_{\alpha^{-1}} \tilde{x}^0)} \nonumber \\
		& \leq \alpha^{-d} \alpha^{-\tau}\, B(\min \{t,\alpha^{-\tau} T^\star\}) \leq \alpha^{-d} B(\min \{t,\alpha^{-\tau} T^\star\}) \nonumber
	\end{align}
	where we have used that $\alpha^{\tau} > 1$ and $\alpha^{-\tau} < 1$ hold. Furthermore, for a scaling factor $c > 1$ (e.g.~$c = \alpha^{-1}$), we get
	\begin{align}
		x(c^{\tau} T^j;\tilde{x}^0,\Delta_{c^{-1}} u^j(c^{-\tau} \cdot)) & = x(c^{\tau} T^j;\Lambda_{c^{-1}} \Lambda_c \tilde{x}^0,\Delta_{c^{-1}} u^j(c^{-\tau} \cdot)) \nonumber \\
		& \stackrel{\cref{eq:HomogeneityTrajectory}}{=} \Lambda_{c^{-1}} \underbrace{x(T^j;\Lambda_c \tilde{x}^0,u^j)}_{\in \alpha \mathcal{N}_{c_1,c_2}} \nonumber
	\end{align}		
	where $c^{\tau} < 1$ and $c^{-1} < 1$.	This shows that a contraction is attained already at time $\alpha^{-\tau} T^j$ ($c = \alpha^{-1}$). %
	For general initial value, the index $j$ has to be suitably chosen. %
	Hence, an iterative application of this line of reasoning shows the claimed inequality not only for $x^0 \in \alpha^i \mathcal{N}_{c_1,c_2}$, $i \in \mathbb{N} \cup \{0\}$, but indeed for all $x^0 \in K$.
	
	Semiglobal asymptotic stability then follows from \cref{thm:AsymptoticStabilityMPC_nonglobal}.
\end{proof}

\begin{remark} \label{rem:conditions}
	As already mentioned after the proof of \cref{thm:AsymptoticStabilityMPC_nonglobal}, %
	\cref{ass:GrowthCondition_nonglobal} is a sufficient condition for stability that is particularly suited for the class of homogeneous problems we investigated in this section. %
	However, there are other types of conditions that may be advantageous for other settings. %
	In this remark we discuss two of them that apply when \cref{ass:GrowthCondition_nonglobal} is violated.

	(i) In discrete time, \cite[Theorem~6.37]{GrunPann17} states that if there is a function $\rho\in\mathcal{K}_\infty$ such that the inequality
	\begin{equation} 
		V_t(x) \le \rho(\ell^\star(x)) \label{eq:cond_alt1}
	\end{equation}
	holds for all $(t,x)\in\R_{\ge 0}\times \R^n$, then the MPC closed loop is semiglobally practically asymptotically stable. %
	If this result were also available in continuous time, then we expect that it provides a way to handle homogeneous systems and running costs with positive degree. %
	In fact, for such systems we would even expect $\rho$ to be bounded from above by an affinely linear function, %
	due to which we conjecture that global practical asymptotic stability can be concluded. %
	Unfortunately, the derivation of these results in continuous time is a nontrivial task and beyond the scope of this paper. It will thus be addressed in future research.

	(ii) Both \cref{ass:GrowthCondition_nonglobal} and inequality~\cref{eq:cond_alt1} require that $\sup_{t\ge 0} V_t(x)<\infty$ for all $x\in\R^n$. %
	However, stability of the closed loop may also hold if $V_t(x)$ grows unboundedly for $t$ tending to infinity. %
	A sufficient condition for asymptotic stability that applies in this case is the following: %
	we assume the existence of a continuous function $C:\R_{\ge 0}\times\R^n\to[0,1]$ that is nonincreasing in $t$ with $C(t,x) < 1$ for all $x\in\R^n$ and $t>0$, %
	such that 
	\begin{equation} 
		V_t(x) \le t C(t,x)\ell^\star(x) \label{eq:cond_alt2}
	\end{equation}
	holds for all $(t,x)\in\R_{\ge 0}\times \R^n$. One easily checks that this condition is satisfied, e.g., for the 1d example 
	\[ 
		\dot x(t) = -|x(t)|(x(t)+u(t)), \qquad \ell(x,u) = |x|+|u|.
	\]
	Abbreviating the optimal trajectory by $x^\star(t) := x(t;x^0,u^\star)$, condition~\cref{eq:cond_alt2} implies 
	\begin{equation} 
		\ell^\star(x^\star(T)) = \inf_{t\in[0,T]}\ell^\star(x^\star(t)). \label{eq:ellinf}
	\end{equation}
	To see this, assume that \cref{eq:ellinf} does not hold. %
	Then there is $\eps>0$ and a $\tau\in[0,T)$ with $\ell^\star(x^\star(\tau)) \le \ell^\star(x^\star(T))-\eps$. %
	Chosing $\tau$ maximal with this property, it follows that $\ell^\star(x^\star(t))\ge \ell^\star(x^\star(T))-\eps \ge \ell^\star(x^\star(\tau))$ for all $t\in[\tau,T]$. %
	By the optimality principle we obtain $V_{T-\tau}(x^\star(\tau)) = \int_{\tau}^T\ell(x^\star(t),u^\star(t))dt \ge (T-\tau) \ell^\star(x^\star(\tau))$, %
	which contradicts \cref{eq:cond_alt2} and thus proves~\cref{eq:ellinf}.

	Now consider the value $J_\delta(x) := \int_0^\delta \ell(x^\star(t),u^\star(t))\, \mathrm{d}t$. %
	Since $\ell \ge 0$ holds, there exists $\tau^\star\in[0,\delta]$ with $\ell^\star(x^\star(\tau^\star)) \le J_\delta(x)/\delta$. %
	Using~\cref{eq:cond_alt2} and~\cref{eq:ellinf}, this implies
	$V_\delta(x^\star(T)) \le C(\delta,x^\star(T))\delta \ell^\star(x^\star(T)) \le C(\delta,x^\star(T)) J_\delta(x)$. From this we can conclude
	\begin{eqnarray*} 
		V_T(x^\star(\delta)) & \le & V_{T-\delta}(x^\star(\delta)) + V_\delta(x^\star(T))\\
		& = & V_T(x) - J_\delta(x) + V_\delta(x^\star(T))\\
		& \le & V_T(x) - (1-C(\delta,x^\star(T)))J_\delta(x) \; < \; V_T(x) 
	\end{eqnarray*}
	for $x\ne 0$. Thus, $V_T$ is a Lyapunov function for the closed loop and asymptotic stability of the closed loop follows.
\end{remark}

\section{Systems with globally asymptotically null controllable homogeneous approximation}\label{sec:HomogeneousApproximation}

In this section, we consider systems with homogeneous approximation at the origin in the following sense \cite{AndrPral08}.
\begin{definition}\label{def:homapp}
	Consider two controlled vector fields $f,h:\R^n\times\R^m\to\R^n$, where $h$ is homogeneous with parameters $r_i>0$, $s_j>0$ and $\tau\in(-\min_i r_i,\infty)$. %
	Then $h$ is called a homogeneous approximation of $f$ near the origin, if the relation
\[ f(x,u) = h(x,u) + R(x,u) \]
holds with
\begin{equation}\label{InequalityHomogeneousApproximation}
	| R_i(\Lambda_\alpha x, \Delta_\alpha u)| \le M \alpha^{r_i+\tau+\eta}
\end{equation}
for all $x,u$ with $\|x\|\le \rho$, $\|u\|\le \rho$ and all $\alpha\in(0,1]$, where $\rho, M, \eta$ are positive constants.
\end{definition}

In case the approximation~$h$ is globally asymptotically null controllable, %
we will establish cost controllability for the system governed by~$f$ %
if the stage cost~\cref{eq:StageCosts} are employed in the sense of \cref{ass:GrowthCondition_nonglobal}, where $K$ is a neighborhood of the origin. %
As a consequence, by means of \cref{thm:AsymptoticStabilityMPC_nonglobal} we obtain local asymptotic stability of the origin for the MPC closed loop. %

Our main tool for deriving these results are homogeneous control Lyapunov functions, whose existence was established in~\cite{Grun00}. %
We first show preliminary results in \cref{sec:robdini}, i.e.\ 
in particular robustness of homogeneous Lyapunov functions. 
Then, we generalize \cref{thm:CostControllability} ($\tau = 0$) and its extension \cref{TheoremCostControllabilityHomogeneousNegative} ($\tau < 0$) before our findings are illustrated by means of the nonholonomic robot example, a variant of Brockett's famous nonholonomic integrator.

\subsection{Properties of Homogeneous Lyapunov Functions}\label{sec:robdini}

In this section we assume that approximation~$h$ is globally asymptotically null controllable.
Then \cite[Theorem 3.5]{Grun00} establishes the existence of a homogeneous control Lyapunov function $V:\R^n\to\R_{\ge 0}$, which
\begin{enumerate}
	\item [1)] is continuous on $\R^n$ and Lipschitz on $\R^n \setminus \{0\}$,
	\item [2)] obeys the equality $V(\Lambda_\alpha x) = \alpha^{2k} V(x)$, and
	\item [3)] satisfies the inequality $\min_{v\in \mathrm{co}\,h(x,W_x)} DV(x;v) \le -\mu N(x)^\tau V(x)$,
\end{enumerate}
where $k := \min_{i=1,\ldots,n} r_i$, $\mu>0$, $W_x =\Delta_{N(x)}U$ for a suitable compact set $U\subset\R^m$,
and $DV(x;v)$ denotes the lower Dini derivative
\[
	DV(x;v) := \liminf_{t\searrow 0,v'\to v}\frac{1}{t}\left( V(x+tv')-V(x)\right).
\]
Next, we show an auxiliary result. Namely, that a suitably scaled version of the homogeneous Lyapunov function~$V$ is locally Lipschitz.
\begin{lemma}\label{LemmaLipschitzContinuity}
	Let $p$ be defined by $\max_{i=1,\ldots,n} r_i$ and let the homogeneous Lyapunov function $V: \mathbb{R}^n \rightarrow \mathbb{R}_{\geq 0}$ satisfy Properties 1) and 2). %
	Then, for each compact set $K \subset \R^n$, there is $L > 0$ %
	such that $\hat{V} := V^{p/(2k)}$ is Lipschitz continuous on $K$ with Lipschitz constant $L$.
\end{lemma}
\begin{proof} A straightforward computation yields that $\hat{V}$ satisfies the (in)equalities
	\begin{equation}
		\hat{V}(\Lambda_\alpha x) = \alpha^p \hat{V}(x) \label{eq:Vhom}
	\end{equation}
	and
	\begin{equation}
		\min_{v \in \mathrm{co}\,h(x,W_x)} D\hat{V}(x;v) \le -\sigma N(x)^\tau \hat{V}(x) \label{eq:Vdecay}
	\end{equation}
	with $\sigma := p \mu/(2k) > 0$. Moreover, Property 1) is preserved. Based on this observation, we can prove the improved Lipschitz property of $\hat{V}$.

	Throughout the proof we use the $1$-norm $\|\cdot\|=\|\cdot\|_1$. %
	Because of the homogeneity, it suffices to show that $\hat{V}$ is Lipschitz on the compact set $K := \{x \in \mathbb{R}^n\,|\, N(x) \leq 1\}$. %
	To this end, first observe that the Lipschitz continuity on $\R^n\setminus\{0\}$ %
	implies the existence of a Lipschitz constant $L_1>0$ on the set $K_1 := \{x\in\R^n\,|\, 1/2 \le N(x) \le 1\}$. %
	Moreover, by continuity of $\hat{V}$ there exists a constant $C_1>0$ such that for all $\hat x\in\R^n$ with $N(\hat x)=1$ and all $\hat y\in \R^n$ with $N(\hat y)\le 1/2$ %
	the inequality $| \hat{V}(\hat x) - \hat{V}(\hat y) | \le C_1$ holds. %
	Together with the fact that there is a constant $C_2>0$ with $\|\hat x-\hat y\|\ge C_2$ for all such $\hat x$ and $\hat y$, for $L_2=C_1/C_2$ %
	we obtain
		$| \hat{V}(\hat x) - \hat{V}(\hat y) | \leq C_1 \leq C_1/C_2 \| \hat{x} - \hat{y} \| = L_2 \| \hat{x} - \hat{y} \|$. %
	This yields that, for all $\hat{x} \in \R^n$ with $N(\hat{x}) = 1$ and all $\hat{y} \in K$, the inequality
	\begin{equation}
		| \hat{V}(\hat x) - \hat{V}(\hat y) | \le L \|\hat x-\hat y\| \label{eq:lip1}
	\end{equation}
	holds with $L:=\max\{L_1,L_2\}$.

	Now consider two arbitrary points $x, y\in K$. %
	Without loss of generality assume $N(x)\ge N(y)$. %
	Let $\alpha:=N(x)\le 1$ and define $\hat x := \Lambda_\alpha^{-1} x$ and $\hat y := \Lambda_\alpha^{-1} y$. %
	Then we obtain
	\[
		\|\hat x - \hat y\| = \sum_{i=1}^n |\hat x_i - \hat y_i| %
		= \sum_{i=1}^n \alpha^{-r_i} |x_i - y_i| \leq \alpha^{-p} \sum_{i=1}^n |x_i - y_i| = \alpha^{-p}\|x-y\|.
	\]
	Observing that $N(\hat x)=1$ and $N(\hat y)\le 1$, we can apply~\cref{eq:lip1} in order to conclude
	\begin{eqnarray*}
		| \hat{V}(x) - \hat{V}(y) | & = & | \hat{V}(\Lambda_\alpha \hat x) - \hat{V}(\Lambda_\alpha \hat y) | = \alpha^p | \hat{V}(\hat x) - \hat{V}(\hat y)|\\
		& \leq & \alpha^p L \|\hat x - \hat y\| \leq \alpha^p L \alpha^{-p} \| x - y\| = L \| x - y\|,
	\end{eqnarray*}
	which shows the claim.
\end{proof}

Based on the Lipschitz continuity of the suitably scaled homogeneous Lyapunov function, we can now derive a robustness statement for the Dini derivative of the Lyapunov function along the vector field $f$ that is homogeneously approximated by~$h$. %
\begin{proposition}\label{prop:Vdecayapprox}
	Consider a vector field~$f$ with homogeneous approximation~$h$ according to \cref{def:homapp} and %
	a homogeneous control Lyapunov function $V: \R^n \to \R_{\ge 0}$ satisfying Properties 2) and 3). %
	Then, for each $\delta \in (0,1)$, there exists $\varepsilon > 0$ such that the inequality
	\begin{equation}
		\min_{v\in \mathrm{co}\,f(x,W_x)} DV(x;v) \le - \delta \sigma N(x)^\tau V(x). \label{eq:Vdecayapprox}
	\end{equation}
	holds for all $x \in \R^n$ with $N(x) \le \varepsilon$.
\end{proposition}
\begin{proof}
	At the beginning of the proof of \cref{LemmaLipschitzContinuity} %
	we made the observation that scaling of the homogeneous Lyapunov function does not change Properties 2) and 3). %
	Analogously, it can be seen that once inequality~\cref{eq:Vdecayapprox} is established for $\hat{V} := V^{p/(2k)}$, %
	a straightforward computation shows that it also holds (with $\delta\sigma$ replaced by $\delta\sigma q$) for $\hat{V}^q$, $q>0$,  and, thus, in particular for $V$. %
	Hence, it is sufficient to prove inequality~\cref{eq:Vdecayapprox} for $\hat{V}$.
	In what follows, we again write $V$ instead of $\hat{V}$ in order to simplify the notation within the remainder of the proof. %

	From the form of $f$ it follows that for all $\alpha\in(0,1]$ and each $v_h\in \mathrm{co}\,h(\Lambda_\alpha x,W_{\Lambda_\alpha x})$ %
	there exists $v\in \mathrm{co}\,f(\Lambda_\alpha x,W_{\Lambda_\alpha x})$ with $v=v_h + v_R(v_h,x,\alpha)$, where
	\begin{equation}
		\left|[v_R(v_h,x,\alpha)]_i\right| \le M \alpha^{r_i+\tau+\eta}, \label{eq:vR}
	\end{equation}
	provided $x$ is sufficiently close to $0$. We choose $\nu>0$ small enough such that \cref{eq:vR} holds whenever $N(x)\le \nu$. %
	inequality~\cref{eq:vR} implies that $v_R(v_h,x,\alpha)$ can be written as
	\[
		v_R(v_h,x,\alpha) = \alpha^{\tau+\eta}\Lambda_\alpha \hat v_R
	\]
	for a vector $\hat v_R = \hat v_R(v_h,x,\alpha)\in\R^n$ with $\left|[\hat v_R]_i\right| \le M$. %
	Any sequence $v'\to v$ can thus be written as $v' = v_h' + \alpha^{\tau+\eta} \Lambda_\alpha \hat v_R$ with $v_h'\to v_h$. %
	From the Lipschitz continuity of $V$ we obtain
	\begin{eqnarray*}
		|V(\Lambda_\alpha x+tv') - V(\Lambda_\alpha x+tv_h')| & \le & \alpha^p|V(x+t\Lambda_\alpha^{-1} v') - V(x+t\Lambda_\alpha^{-1} v_h')|\\
		& \le & \alpha^p tL \| \Lambda_\alpha^{-1} \alpha^{\tau+\eta}\Lambda_\alpha \hat v_R \| = \alpha^{p+\tau+\eta} tL \| \hat v_R \|.
	\end{eqnarray*}
	For the Dini-derivative this implies
	\[
		\min_{v\in \mathrm{co}\,f(\Lambda_\alpha x,W_{\Lambda_\alpha x})} DV(\Lambda_\alpha x;v) %
		\le \min_{v_h\in \mathrm{co}\,h(\Lambda_\alpha x,W_{\Lambda_\alpha x})} DV(\Lambda_\alpha x;v) + \alpha^{p+\tau+\eta} M_R.
	\]
	with $M_R$ being a bound on $\|\hat v_R\|$.

	Now let $m_V := \min \{ V(x) \, |\, N(x) = \nu\}$. Then \cref{eq:Vhom} implies
	\[
		\alpha^{p+\tau+\eta}M_R \le (1-\delta)\sigma\alpha^{p+\tau} V(x) \nu^\tau
	\]
	for all $x\in\R^n$ with $N(x)=\nu$, %
	whenever $\alpha \le \alpha_\delta := ((1-\delta)\sigma m_V\nu^\tau/M_R)^{1/\eta}$. %
	Given $x\in\R^n$ with $N(x)\le \nu$, setting $\alpha := N(x)/\nu\in(0,1]$ and $\hat x := \Lambda_\alpha^{-1}(x)$ %
	we obtain $N(\hat x) = \nu$. Hence, for $\eps = \alpha_\delta \nu$ and $N(x)\le \eps$ (implying $\alpha \le \alpha_\delta$) we obtain
	\begin{eqnarray*}
		\min_{v\in \mathrm{co}\,f(x,W_{x})} DV(x;v) & = & \min_{v\in \mathrm{co}\,f(\Lambda_\alpha \hat x,W_{\Lambda_\alpha \hat x})} DV(\Lambda_\alpha \hat x;v)\\
		& \le & \min_{v_h\in \mathrm{co}\,h(\Lambda_\alpha \hat x,W_{\Lambda_\alpha \hat x})} DV(\Lambda_\alpha x;v) + \alpha^{p+\tau+\eta} M_R\\
		& \le & - \sigma \underbrace{N(\Lambda_\alpha \hat x)^\tau}_{=\alpha^\tau \nu^\tau}\underbrace{V(\Lambda_\alpha \hat x)}_{=\alpha^p V(\hat x)} %
		+ \underbrace{\alpha^{p+\tau+\eta} M_R}_{\le (1-\delta)\sigma\alpha^{p+\tau} V(\hat x)\nu^\tau} \\
		& \le & - \delta\sigma \alpha^{p+\tau}\nu^\tau V(\hat x) = - \delta\sigma N(\Lambda_\alpha \hat x)^\tau V(\Lambda_\alpha \hat x) = - \delta\sigma N(x)^\tau V(x).
	\end{eqnarray*}
	This shows the claim.
\end{proof}

\subsection{Cost controllability via homogeneous control Lyapunov functions}

We now show that cost controllability can be concluded from the existence of a homogeneous control Lyapunov function for the homogeneous approximation. %
To this end, we use the observation made in the proof of \cite[Theorem 1]{SonS95} %
(based on a similar statement for differential inclusions from \cite{AubiCell84}), %
that inequality~\cref{eq:Vdecayapprox} implies the existence of a measurable control function $u : [0,t] \to \R^m$ for which the integral inequality
\begin{equation}
	V(x(t,x^0,u)) - V(x^0) \le \int_0^t - q \delta \sigma N(x(s;x^0,u))^\tau V(x(s,x^0,u))\,\mathrm{d}s \label{eq:intV}
\end{equation}
and the inclusion $u(s) \in W_{x(s,x^0,u)}$ hold for all almost all $s \geq 0$. %
While the construction in \cite{SonS95} is for global control Lyapunov functions, %
standard Lyapunov function arguments show that it remains valid if the control Lyapunov function is only defined in a neighbourhood of the origin. %
Of course, in this case \cref{eq:intV} will only hold for $x^0$ from a (possibly different) neighbourhood of the origin.

\begin{theorem}
	Consider a vector field $f$ with homogeneous approximation $h$ of degree $\tau \leq 0$ according to \cref{def:homapp}. %
	Assume that the homogeneous system with vector field $h$ is globally asymptotically null controllable to the origin. %
	Then there exists a neighborhood $\mathcal{N}$ of the origin such that \cref{eq:GrowthCondition} is satisfied with $K = \mathcal{N}$, %
	i.e.\ the cost controllability
	is satisfied for the vector field $f$ and $B_K(t)$ is uniformly bounded. %
	In particular, given $\delta>0$, for any sufficiently large $T>0$ the origin is locally asymptotically stable for the MPC closed loop.
\end{theorem}
\begin{proof} By \cite[Theorem 3.5]{Grun00} and \cref{prop:Vdecayapprox} %
	there exists a control Lyapunov function $V$ satisfying inequality~\cref{eq:Vdecayapprox} and, by suitably scaling, we have 
	\begin{equation}
		V(\Lambda_\alpha x) = \alpha^{d} V(x) \label{eq:Vhomd}
	\end{equation}
	and the inequality
	\begin{equation}
		\min_{v\in \mathrm{co}\,f(x,W_x)} D V(x;v) \leq 
		- \gamma N(x)^\tau V(x)\label{eq:Vdecayapproxd}
	\end{equation}
	for $\gamma > 0$ in a neighbourhood $\mathcal{N}$ of the origin. %
	From equation~\cref{eq:Vhomd} it follows that there exists $C_1>0$ with
	\begin{equation}\label{eq:EstimateHomogeneousLyapunovFunction}
		V(x) \le C_1 \inf_{u \in \mathbb{R}^m} \ell(x,u)
	\end{equation}
	for all $x\in\R^n$. %
	Now pick an initial condition $x^0$ from the neighbourhood $\mathcal{N}$ of the origin on which \cref{eq:intV} with $q = d/p$ holds. %
	Then $V$ decreases along the solution which implies that $V(x(t,x^0,u))$ and thus $N(x(t,x^0,u))$ are bounded uniformly in $t$ and $x^0$. %
	As a consequence, there exists $C_2>0$, independent of $x^0$, such that $N(x(t,x^0,u))^\tau \ge C_2$ for all $t\ge 0$ using $\tau \leq 0$.

	Then, along the solution in \cref{eq:intV}, again \cref{eq:Vhomd} and the inclusion $u(s) \in W_{x(s,x^0,u)}$ %
	(which implies $N(u(s)) \leq D N(x(s,x^0,u))$ for an appropriate constant $D>0$) %
	yield the existence of $C_3>0$ with
	\[
		\ell(x(s,x^0,u),u(s)) \le C_3 V(x(s,x^0,u),u(s))
	\]
	for almost all $s\in (0,t)$. Thus, using \cref{eq:intV} we obtain
	\begin{eqnarray*}
		V_t(x^0) & \leq & \int_0^t \ell(x(s,x^0,u),u(s))\,\mathrm{d}s \leq \int_0^t C_3 V(x(s,x^0,u),u(s))\,\mathrm{d}s \\
		& \le & \frac{C_3}{C_2 \gamma} \Big( V(x^0)-V(x(t,x^0,u)) \Big) \le \frac{C_3}{C_2\gamma}V(x^0) \\
		& \stackrel{\cref{eq:EstimateHomogeneousLyapunovFunction}}{\leq} & \frac{C_1 C_3}{C_2\gamma}\inf_{u \in \mathbb{R}^m} \ell(x^0,u).
	\end{eqnarray*}
	This shows the first claim for $C=C_1C_3/(C_2\gamma)$. %
	Asymptotic stability of the origin for the MPC closed loop then follows from \cref{thm:AsymptoticStabilityMPC_nonglobal}.
\end{proof}

\subsection{Example: Nonholonomic Robot}

We illustrate our findings by means of the following example.
\begin{example}\label{ExampleNonholonomicMobileRobot}
	We consider the nonholonomic mobile robot given by
	\begin{equation}\label{eq:NonholonomicMobileRobot}
		\dot{x}(t) = f(x(t),u(t)) := \left( \begin{array}{c}
			\cos(x_3(t)) \\
			\sin(x_3(t)) \\
			0
		\end{array} \right) u_1(t) + \left( \begin{array}{c}
			0 \\
			0 \\
			1
		\end{array} \right) u_2(t),
	\end{equation}
	which is a control affine system, i.e.~$f(x,u) = f_0(x) + \sum_{i=1}^m u_i f_i(x)$ %
	with mappings $f_0,f_1,\ldots,f_m: \mathbb{R}^n \rightarrow \mathbb{R}^n$. %
	The system dynamics~\cref{eq:NonholonomicMobileRobot} describe a \textit{driftless} system since $f_0(x) \equiv 0$.
\end{example}

Let us recall that, as shown by Brockett in \cite{1983-Brockett-proccedings} (see also \cite{1990-Coron-SCL} and \cite[Theorem 11.1]{Coro07}), %
the control system~\cref{eq:NonholonomicMobileRobot} cannot be locally asymptotically stabilizable by means of continuous feedback laws. %
We show that the proposed MPC method allows to locally asymptotically stabilize the control system~\cref{eq:NonholonomicMobileRobot}. %
To this end, we firstly state the homogeneous approximation and show that the approximation property, i.e.\ Inequality~\cref{InequalityHomogeneousApproximation}, is satisfied. %
Note that the parameters $r_i$ and $s_j$ are not unique because the homogeneity condition~\cref{NotationHomogeneityCondition} remains satisfied if they are multiplied by an arbitrary positive scalar.
\begin{proposition}\label{PropositionNonholonomicRobot}
	Inequality~\cref{InequalityHomogeneousApproximation} holds for the homogeneous approximation of \cref{ExampleNonholonomicMobileRobot} at the origin, which is given by
	\begin{equation}\label{eq:NonholonomicMobileRobotHomogeneousApproximation}
		\dot{x}(t) = h(x(t),u(t)) := \left( \begin{array}{c}
			1 \\
			x_3(t) \\
			0
		\end{array} \right) u_1(t) + \left( \begin{array}{c}
			0 \\
			0 \\
			1
		\end{array} \right) u_2(t)
	\end{equation}
	with $\tau = 0$, $r_1 = r_3 = s_1 = s_2 = 1$, and $r_2 = 2$.
\end{proposition}	
\begin{proof}
	For $i = 3$, the approximation error is zero. %
	For $i = 1$, we have to show the inequality $| (\cos( \alpha^{r_3} x_3 ) - 1) \alpha^{r_1} u_1 | \leq M_1 \alpha^{r_1 + \tau + \eta_1}$. %
	Using $r_1 = r_3 = 1$, we have $| \cos( \alpha x_3 ) - 1 | \leq \alpha^2 x_3^2 / 2$. %
	Hence, $| (\cos( \alpha^{r_3} x_3 ) - 1) \alpha^{r_1} u_1 | $ is smaller than $\alpha^{3} \rho^3 / 2$, which shows the assertion for $M_1 := \rho^3/2$ and $\eta_1 \in [0,2]$. %
	For $i = 2$, we analogously derive the estimate
	\[	
		| \sin(\alpha^{r_3}x_3) - \alpha^{r_3}x_3 | \alpha^{r_1} |u_1| \leq \alpha^{4}|x_3|^3 |u_1|/6 \leq \alpha^{4} \rho^4 / 6
	\]
	and, thus, get the estimate for $M_2 := \rho^4/6$ and $\eta_2 \in [0,2]$. Overall, the desired property holds with $M := \max\{ \rho^3/2, \rho^4/6\}$ and $\eta \in (0,2]$.
\end{proof}	

\noindent For the nonholonomic mobile robot, i.e.\ \cref{ExampleNonholonomicMobileRobot}, MPC does not work  if (purely) quadratic stage cost are used, i.e.\ $\ell(x,u) := x^\top Q x + u^\top R u$ with matrices $Q \in \mathbb{R}^{n \times n}$ and $R \in \mathbb{R}^{m \times m}$ as rigorously shown in~\cite{MullWort17}. Using the proposed stage cost~\cref{eq:StageCosts} reads for the nonholonomic robot (and its homogeneous approximation at the origin) as
\begin{equation}\label{NotationStageCostNonholonomicRobot}
	\ell(x,u) = x_1^4 + x_2^2 + x_3^4 + u_1^4 + u_2^4,
\end{equation}
i.e.\ precisely the stage costs used in~\cite{WMZGMD_NMPC_2015}, for which asymptotic stability of the origin w.r.t.\ the MPC closed loop holds if a sufficiently large prediction horizon~$T$ is employed, cf.~\cite{worthmann2016model}. %
Using the results derived in this paper, %
local asymptotic stability of the origin is a direct corollary of the simple calculations presented in \cref{PropositionNonholonomicRobot}. %
Moreover, we conjecture that the observations w.r.t.\ \textit{essentially using the homogeneous approximation} may significantly facilitate (and improve) the quantitative estimates on the length of the prediction horizon.
\begin{remark}
	Note that the particular definition of $d$ has not been used so far. This explains, e.g., why the stage cost~\cref{NotationStageCostNonholonomicRobot} for the nonholonomic robot may be replaced by
	\begin{equation}\nonumber
		\ell(x,u) = \sum_{i=1}^n q_{x_i} |x_i|^{d/r_i} + \sum_{j=1}^m q_{u_j} |u_j|^{d/s_j}
	\end{equation}
	with arbitrary $d \in (0,\infty)$ where $q_{x_i}$, $i \in \{1,2,\ldots,n\}$, and $q_{u_j}$, $j \in \{1,2,\ldots,m\}$ are positive weighting factors, see \cite[Subsection~2.2]{WortRebl14} for an example.\footnote{In this reference, Brockett's nonholonomic integrator example was considered. However, this example is (locally) equivalent to the nonholonomic robot (unicycle), see \cite[pp.\,83-89]{liberzon2003switching}.}
\end{remark}
\begin{remark}
	Other homogeneous approximations can also been used to locally stabilize in small-time the control system~\cref{eq:NonholonomicMobileRobot} by means of periodic time-varying feedback laws as shown in \cite{AndrCoro19}.
\end{remark}

\section{Conclusions}\label{sec:conclusions}

For nonlinear systems with stabilizable linearization at the origin, %
MPC \textit{works} (with quadratic stage cost) for sufficiently large prediction horizon (without adding stabilizing terminal costs and/or constraints). %
We have generalized this methodology to nonlinear systems with globally asymptotically null controllable homogeneous approximation at the desired set point. %
To this end, we have set up a general framework for rigorously showing cost controllability %
and, thus, local asymptotic stability of the origin w.r.t.\ the MPC closed loop for systems by designing the stage cost based on the homogeneous approximation.

A major advantage of the presented approach is its simplicity from the user's perspective: %
calculating the homogeneous approximation, checking (local) asymptotic stabilizability using the already existing results from the literature, %
and finally running MPC with the stage costs based on the coefficients from the homogeneity  definition.

\bibliographystyle{plain}
\bibliography{Homogeneity_url}

\def\cprime{$'$} \newcommand{\SortNoop}[1]{}
\begin{thebibliography}{10}

\bibitem{AndrCoro19}
{\SortNoop{Andr\'{e}a-Novel, B.}}{B. d'Andr\'{e}a-Novel}, J.-M. Coron, and
  W.~Perruquetti.
\newblock Small-time stabilization of nonholonomic or underactuated mechanical
  systems: the unicycle and the slider examples.
\newblock {\em Preprint}, 2019.
\newblock \url{https://hal.archives-ouvertes.fr/hal-02140549}.

\bibitem{AndrPral08}
V.~Andrieu, L.~Praly, and A.~Astolfi.
\newblock Homogeneous approximation, recursive observer design, and output
  feedback.
\newblock {\em SIAM Journal on Control and Optimization}, 47(4):1814--1850,
  2008.

\bibitem{AubiCell84}
J.-P. Aubin and A.~Cellina.
\newblock {\em Differential Inclusions}.
\newblock Springer, 1984.

\bibitem{BaccRosi06}
A.~Bacciotti and L.~Rosier.
\newblock {\em Liapunov functions and stability in control theory}.
\newblock Communications and Control Engineering Series. Springer-Verlag,
  Berlin, second edition, 2005.

\bibitem{BhatBern05}
S.P. Bhat and D.S. Bernstein.
\newblock Geometric homogeneity with applications to finite-time stability.
\newblock {\em Mathematics of Control, Signals and Systems}, 17(2):101--127,
  2005.

\bibitem{BoccGrun14}
A.~Boccia, L.~Gr{\"u}ne, and K.~Worthmann.
\newblock Stability and feasibility of state constrained mpc without
  stabilizing terminal constraints.
\newblock {\em {Systems \& Control Letters}}, 72:14--21, 2014.

\bibitem{1983-Brockett-proccedings}
R.W. Brockett.
\newblock Asymptotic stability and feedback stabilization.
\newblock In {\em Differential geometric control theory ({H}oughton, {M}ich.,
  1982)}, volume~27 of {\em Progr. Math.}, pages 181--191. Birkh\"{a}user
  Boston, Boston, MA, 1983.

\bibitem{ChenAllg98}
H.~Chen and F.~Allg{\"o}wer.
\newblock A quasi-infinite horizon nonlinear model predictive control scheme
  with guaranteed stability.
\newblock {\em Automatica}, 34(10):1205--1217, 1998.

\bibitem{1990-Coron-SCL}
J.-M. Coron.
\newblock A necessary condition for feedback stabilization.
\newblock {\em Systems Control Lett.}, 14(3):227--232, 1990.

\bibitem{Coro07}
J.-M. Coron.
\newblock {\em Control and nonlinearity}, volume 136 of {\em Mathematical
  Surveys and Monographs}.
\newblock American Mathematical Society, Providence, RI, 2007.

\bibitem{CoroPral91}
J.-M. Coron and L.~Praly.
\newblock Adding an integrator for the stabilization problem.
\newblock {\em Systems \& Control Letters}, 17(2):89--104, 1991.

\bibitem{ForbPatw15}
M.G. Forbes, R.S. Patwardhan, H.~Hamadah, and R.B. Gopaluni.
\newblock Model predictive control in industry: Challenges and opportunities.
\newblock {\em IFAC-PapersOnLine}, 48(8):531--538, 2015.

\bibitem{Grun00}
L.~Gr{\"u}ne.
\newblock Homogeneous state feedback stabilization of homogenous systems.
\newblock {\em SIAM Journal on Control and Optimization}, 38(4):1288--1308,
  2000.

\bibitem{Grun09}
L.~Gr{\"u}ne.
\newblock Analysis and design of unconstrained nonlinear {M}{P}{C} schemes for
  finite and infinite dimensional systems.
\newblock {\em SIAM Journal on Control and Optimization}, 48:1206--1228, 2009.

\bibitem{GrunPann17}
L.~Gr{\"u}ne and J.~Pannek.
\newblock Nonlinear model predictive control.
\newblock In {\em Nonlinear Model Predictive Control}. Springer, 2017.

\bibitem{GrunPann10}
L.~Gr{\"u}ne, J.~Pannek, M.~Seehafer, and K.~Worthmann.
\newblock Analysis of unconstrained nonlinear mpc schemes with time varying
  control horizon.
\newblock {\em SIAM Journal on Control and Optimization}, 48(8):4938--4962,
  2010.

\bibitem{Herm95}
H.~Hermes.
\newblock Homogeneous feedback controls for homogeneous systems.
\newblock {\em Systems \& Control Letters}, 24(1):7--11, 1995.

\bibitem{Kaws88}
M.~Kawski.
\newblock Stability and nilpotent approximations.
\newblock In {\em Proc. 27th IEEE Conf. on Decision and Control}, pages
  1244--1248, 1988.

\bibitem{Kawsk90}
M.~Kawski.
\newblock Homogeneous stabilizing feedback laws.
\newblock {\em Control Theory and advanced technology}, 6(4):497--516, 1990.

\bibitem{Kell14}
C.M. Kellett.
\newblock A compendium of comparison function results.
\newblock {\em Mathematics of Control, Signals, and Systems}, 26(3):339--374,
  2014.

\bibitem{KouvCann16}
B.~Kouvaritakis and M.~Cannon.
\newblock {\em Model Predictive Control}.
\newblock Springer, Basel, 2016.

\bibitem{Lee11}
J.H. Lee.
\newblock Model predictive control: Review of the three decades of development.
\newblock {\em International Journal of Control, Automation and Systems},
  9(3):415, 2011.

\bibitem{liberzon2003switching}
D.~Liberzon.
\newblock Switching in systems and control, ser. systems \& control:
  Foundations \& applications.
\newblock {\em Birkh\"{a}user}, 2003.

\bibitem{Mori99}
P.~Morin, J.-B. Pomet, and C.~Samson.
\newblock Design of homogeneous time-varying stabilizing control laws for
  driftless controllable systems via oscillatory approximation of lie brackets
  in closed loop.
\newblock {\em SIAM Journal on Control and Optimization}, 38(1):22--49, 1999.

\bibitem{MoriSams96}
P.~Morin and C.~Samson.
\newblock Time-varying exponential stabilization of chained form systems based
  on a backstepping technique.
\newblock In {\em Proc. 35th IEEE Conf. Decision and Control}, volume~2, pages
  1449--1454, 1996.

\bibitem{MullWort17}
M.A. M{\"u}ller and K.~Worthmann.
\newblock Quadratic costs do not always work in {MPC}.
\newblock {\em Automatica}, 82:269--277, 2017.

\bibitem{RawlMayn17}
J.B. Rawlings, D.Q. Mayne, and M.M. Diehl.
\newblock {\em Model Predictive Control: Theory, Computation, and Design}.
\newblock Nob Hill Publishing, 2017.

\bibitem{ReblAllg12}
M.~Reble and F.~Allg{\"o}wer.
\newblock Unconstrained model predictive control and suboptimality estimates
  for nonlinear continuous-time systems.
\newblock {\em Automatica}, 48(8):1812--1817, 2012.

\bibitem{Rosi92}
L.~Rosier.
\newblock Homogeneous lyapunov function for homogeneous continuous vector
  field.
\newblock {\em Systems \& Control Letters}, 19(6):467--473, 1992.

\bibitem{SonS95}
E.D. Sontag and H.J. Sussmann.
\newblock Nonsmooth control-{L}yapunov functions.
\newblock In {\em Proc. 34th IEEE Conf. Decision and Control}, New Orleans, LA,
  USA, 1995.

\bibitem{TunaMess06}
S.E. Tuna, M.J. Messina, and A.R. Teel.
\newblock Shorter horizons for model predictive control.
\newblock In {\em Proc. Amer. Control Conf.}, pages 863--868, Minneapolis, MN,
  USA, 2006.

\bibitem{Wort11}
K.~Worthmann.
\newblock {\em {S}tability {A}nalysis of {U}nconstrained {R}eceding {H}orizon
  {C}ontrol {S}chemes}.
\newblock PhD thesis, University of Bayreuth, 2011.

\bibitem{WMZGMD_NMPC_2015}
K.~Worthmann, M.W. Mehrez, M.~Zanon, R.G. Gosine, G.K.I. Mann, and M.~Diehl.
\newblock Regulation of differential drive robots using continuous time mpc
  without stabilizing constraints or costs.
\newblock {\em IFAC-PapersOnLine}, 48(23):129--135, 2015.

\bibitem{worthmann2016model}
K.~Worthmann, M.W. Mehrez, M.~Zanon, G.K.I. Mann, R.G. Gosine, and M.~Diehl.
\newblock Model predictive control of nonholonomic mobile robots without
  stabilizing constraints and costs.
\newblock {\em IEEE Transactions on Control Systems Technology},
  24(4):1394--1406, 2016.

\bibitem{WortRebl14}
K.~Worthmann, M.~Reble, L.~Gr\"{u}ne, and F.~Allg\"{o}wer.
\newblock The role of sampling for stability and performance in unconstrained
  nonlinear model predictive control.
\newblock {\em SIAM Journal on Control and Optimization}, 52(1):581--605, 2014.

\end{thebibliography}
\end{document}